\documentclass[11pt]{amsart}
\pdfoutput=1
\usepackage{amsmath, amssymb, amsthm, graphicx, subfig, mathrsfs, calc, xspace,
  mathtools, ifthen, booktabs, longtable, comment, setspace, verbatim,
  placeins, abraces}
\usepackage[colorlinks=true, urlcolor=blue, citecolor=red]{hyperref}
\newlength{\figWidthA}
\newlength{\figWidthB}

\newcommand{\xth}[1][th]{\textsuperscript{#1}\xspace}
\newcommand{\set}[2]{\ifthenelse{\equal{#2}{}}{\{#1\}}{\{#1 \mid #2\}}}
\newcommand{\ol}{\overline}
\newcommand{\s}{\sigma}
\newcommand{\nats}{\mathbb{N}}

\newcommand{\G}{\mathscr{G}}
\newcommand{\V}[1][G]{V(#1)}
\newcommand{\E}[1][G]{E(#1)}
\newcommand{\N}[2][]{\ifthenelse{\equal{#1}{}}{N(#2)}{N_{#1}(#2)}}
\renewcommand{\deg}[2][]{\ifthenelse{\equal{#1}{}}{d(#2)}{d_{#1}(#2)}}
\newcommand{\size}[1]{\##1}

\newcommand{\gamename}[1]{\ifthenelse{\equal{\showinggamename}{false}}%
  {}{^{#1}}}
\newcommand{\showgame}{\def\showinggamename{true}}
\newcommand{\hidegame}{\def\showinggamename{false}}

\hidegame

\newcommand{\pos}[2][\s]{#1_{#2}}
\newcommand{\chips}[3][\s]{#1_{#3}(#2)}
\newcommand{\receiving}[3][\s]{\Phi\gamename{#1}_{#3}(#2)}
\newcommand{\firing}[3][\s]{F\gamename{#1}_{#3}(#2)}
\newcommand{\pfp}[2][\s]{F\gamename{#1}(#2)}
\newcommand{\period}[1][\s]{T\gamename{#1}}
\newcommand{\pfpverts}[2][\s]{\tau\gamename{#1}(#2)}
\newcommand{\pfpedges}[3][\s]{\pi\gamename{#1}(#2,#3)}
\newcommand{\misb}[4][\s]{m\gamename{#1}_{#2}(#3,#4)}
\newcommand{\sctr}[1]{\mathcal{X}(#1)}
\newcommand{\mots}[1][\s]{M\gamename{#1}}
\newcommand{\pfps}[1][\s]{P\gamename{#1}}
\newcommand{\cpfps}[1][\s]{Q\gamename{#1}}

\newtheorem{thm}{Theorem}[section]
\newtheorem{lem}[thm]{Lemma}
\newtheorem{cor}[thm]{Corollary}

\numberwithin{equation}{section}

\begin{document}

\setlength\figWidthA{.72\textwidth}
\setlength\figWidthB{.28\textwidth}

\title[Motors and Impossible Firing Patterns in the PCFG]{Motors and Impossible
  Firing Patterns in the\\Parallel Chip-Firing Game}
\author{Tian-Yi Jiang, Ziv Scully, Yan X Zhang}
\address{Massachusetts Institute of Technology}
\email{jiangty@mit.edu\textrm{, }ziv@mit.edu\textrm{, }yanzhang@post.harvard.edu}
\begin{abstract}
The \emph{parallel chip-firing game} is an automaton on graphs in which
vertices ``fire'' chips to their neighbors. This simple model contains much
emergent complexity and has many connections to different areas of
mathematics. In this work, we study \emph{firing sequences}, which describe
each vertex's interaction with its neighbors in this game. First, we introduce
the concepts of \emph{motors} and \emph{motorized games}. Motors both
generalize the game and allow us to isolate local behavior of the (ordinary)
game. We study the effects of motors connected to a tree and show that
motorized games can be transformed into ordinary games if the motor's firing
sequence occurs in some ordinary game. Then, we completely characterize the
periodic firing sequences that can occur in an ordinary game, which have a
surprisingly simple combinatorial description.

\end{abstract}
\maketitle

\section{Introduction}

\subsection*{Background}
The \emph{parallel chip-firing game}, also known as the \emph{discrete
  fixed-energy sandpile model}, is an automaton on graphs in which vertices
that have at least as many chips as incident edges ``fire'' chips to their
neighbors. In graph theory, it has been studied in relation with the critical
group of graphs~\cite{biggs}. In computer science, it is able to simulate any
two-register machine and is thus universal~\cite{universality}. As a specific
case of the \emph{abelian sandpile model}, which is itself a generalization of
a sandpile model introduced by Bak, Tang, and Weisenfeld~\cite{BTWcriticality,
  BTWmoreCriticality} in the study of self-organized criticality, it has even
more links with other fields.

\subsection*{The Game}
The parallel chip-firing game is played on a graph as follows:
\begin{itemize}
\item At first, a nonnegative integer number of chips is placed on each vertex
  of the graph.
\item The game then proceeds in discrete turns. Each turn, a vertex checks to
  see if it has at least as many chips as incident edges.
\begin{itemize}
\item If so, that vertex \emph{fires}.
\item Otherwise, that vertex \emph{waits}.
\end{itemize}
\item To fire, a vertex passes one chip along each of its edges. All vertices
  that fire in a particular turn do so in parallel.
\item Immediately after firing or waiting, every vertex receives any chips that
  were fired to it.
\end{itemize}
Here we will only consider games on finite, undirected, connected graphs,
though the definition of the game can be easily generalized for arbitrary
multidigraphs. An example game is illustrated in Figure~\ref{example}. Given a
parallel chip-firing game $\s$, we refer to the chip configuration, also called
the position, at a particular time $t \in \nats$ as $\pos{t}$.

\begin{centering}
\begin{figure}[h]
  \subfloat{\includegraphics[width=\figWidthA]{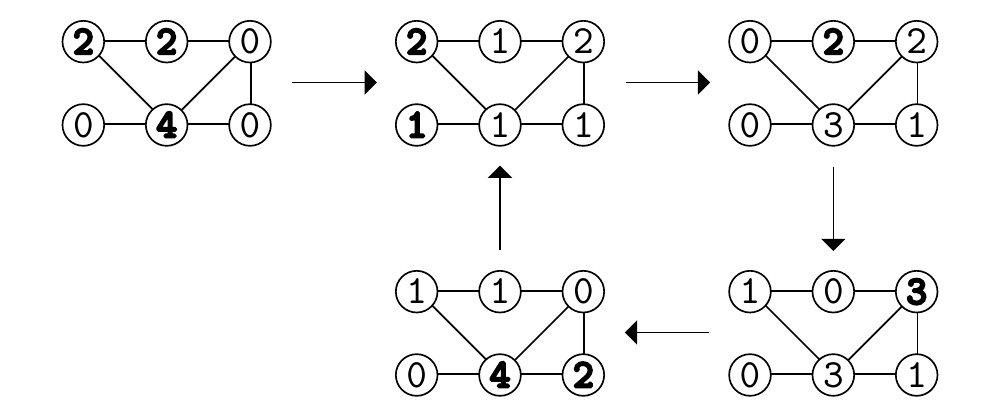}}
  \subfloat{\includegraphics[width=\figWidthB]{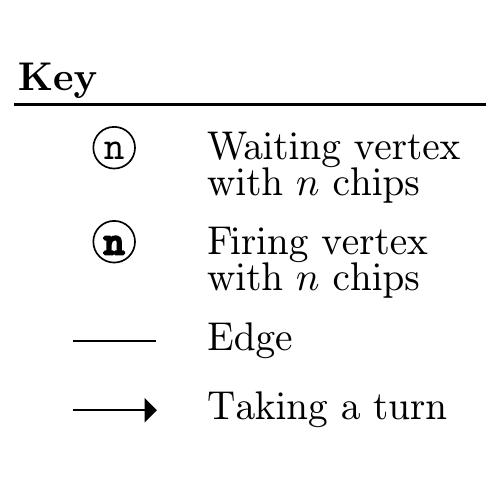}}
  \caption{A parallel chip-firing game. From an initial position in the upper
    left, the game eventually enters a period of length 4.}
  \label{example}
\end{figure}
\end{centering}

The total number of chips on all vertices of the graph is constant throughout a
game, so there are finitely many possible positions in every game. Therefore,
every game eventually reaches a position $\pos{t}$ that is identical to a later
position $\pos{t+p}$ for some $t,p \in \nats$ with $p > 0$. (We write $\pos{t}
= \pos{t+p}$.) The game is deterministic, so $\pos{t+n} = \pos{t+n+p}$ for all
$n \in \nats$. Thus, every parallel chip-firing game is eventually periodic.

In this paper, we concern ourselves with both \emph{firing sequences} and
\emph{periodic firing patterns} of vertices. Each is a binary string
representing whether or not a particular vertex fires or waits each turn. The
sequence covers all times from 0 to infinity, while the periodic pattern covers
just one period.

\subsection*{Previous Work}
The periodicity of the parallel chip-firing game gives rise to two
questions. First, what characteristics of a game and its underlying graph
determine the length of a period? It is known exactly what periods are possible
on certain classes of graphs, such as trees~\cite{bitarGoles}, simple
cycles~\cite{cycle}, the complete graph~\cite{levine}, and the complete
bipartite graph~\cite{jiang}. For these graphs, the maximum period lengths are
bounded by the number of vertices, but Kiwi et al.~\cite{kiwiEtAl} constructed
graphs on which the period of games can grow exponentially with polynomial
increase in the number of vertices. There are also results regarding the total
number of chips in a game. Kominers and Kominers~\cite{kominers} showed that
games with a sufficiently large density of chips must have
period~1. Dall'Asta~\cite{cycle} and Levine~\cite{levine}, in their respective
characterizations of periods on cycles and complete graphs, related the total
number of chips to a game's \emph{activity}, the fraction of turns during which
a vertex fires. The denominator of the activity must divide the period.

Second, we notice that some but not all positions $\pos{t}$ are
\emph{periodic}, satisfying $\pos{t} = \pos{t+p}$ for some positive $p \in
\nats$. What characterizes periodic positions? This problem has not been as
extensively studied. Dall'Asta~\cite{cycle} characterized the periodic
positions of games on cycles.

\subsection*{Our Results}
We hope to advance the understanding of both of these questions through the
study of firing sequences and periodic firing patterns.

After precisely defining the parallel chip-firing game in Section~\ref{pres},
the first half of the paper develops a new tool for studying the chip-firing
game: \emph{motors}, vertices that fire with a regular pattern independent of
normal chip-firing rules. Games with motors are called \emph{motorized
  games}. Motors allow us to study the behavior of subgraphs in ordinary
parallel chip-firing games. In Section~\ref{motors} we show that vertices
always ``follow'' a motor in periodic motorized games on trees. In
Section~\ref{simulatingMotors}, we prove that periodic motorized games can be
transformed into ordinary games as long as the firing sequence of each motor
occurs in an ordinary game.

The second half of the paper characterizes the possible periodic firing
patterns in parallel chip-firing games. Section~\ref{binSeq} briefly steps away
from the game to study certain signed sums of periodic binary sequences. The
result is an inequality applicable to edges of the graph of a parallel
chip-firing game. In Section~\ref{nonclumpiness}, we sum this inequality over
all relevant edges to show that periodic firing patterns with both consecutive
0s and consecutive 1s cannot occur in a parallel chip-firing game. This, along
with an already known construction, fully characterizes the periodic firing
patterns possible in parallel chip-firing games. Finally, in
Section~\ref{corollaries}, we examine some implications of this theorem.

\section{Preliminaries} \label{pres}

\subsection*{Definitions}
A \emph{parallel chip-firing game} $\s$ on a graph $G = (\V, \E)$ is a sequence
$(\pos{t})_{t \in \nats}$ of ordered tuples with natural number elements
indexed by $\V$. Each tuple represents the chip configuration at a particular
turn, where each element of the tuple is the number of chips on the
corresponding vertex. We define the following for all $v \in \V$:
\begin{align*}
  \N{v} &= \set{w \in \V}{\{v,w\} \in \E} \\
  \deg{v} &= \size{\N{v}} \\
  \chips{v}{t} &= \text{number of chips on $v$ in position $\pos{t}$} \\
  \firing{v}{t} &= \begin{cases}
    0 &\text{ if } \chips{v}{t} \leq \deg[]{v} - 1 \\
    1 &\text{ if } \chips{v}{t} \geq \deg[]{v}
  \end{cases} \\
  \receiving{v}{t} &= \sum_{\mathclap{w \in \N{v}}} \firing{w}{t}.
\end{align*}
In a parallel chip-firing game, $\pos{t}$ induces $\pos{t+1}$. For all $v \in
\V$,
\begin{equation}\label{gameDef}
  \chips{v}{t+1} = \chips{v}{t} + \receiving{v}{t} - \firing{v}{t}\deg[]{v},
\end{equation}
so an initial position suffices to define a game on a given graph. When
$\firing{v}{t} = 0$, we say $v$ \emph{waits} at $t$, and when $\firing{v}{t} =
1$, we say $v$ \emph{fires} at $t$.

A position $\pos{t}$ is \emph{periodic} if and only if there exists $p \in \nats$ such
that $\pos{t} = \pos{t+p}$. The minimum such $p$ for which this occurs is the
\emph{period} of $\s$ and is denoted $\period$. Abusing notation slightly, ``a
period'' of a game $\s$ may also refer to a set of times $\{t, t+1, \dots,
t+\period-1\}$, where $\pos{t}$ is periodic. The parallel chip-firing game is
deterministic and there are finitely many possible positions on a given graph
with a given number of chips, so for any game $\s$, there exists $t_0 \in
\nats$ such that $\pos{t}$ is periodic for all $t \geq t_0$. If the initial
position of a game is periodic, we may also call the game itself periodic.

\subsection*{Notation}
\newlength{\tablespace}
\setlength{\tablespace}{.3\baselineskip}
Definitions for invented notation are given in the section indicated in the
last column.\\

\showgame
\begin{centering}
  \begin{tabular}{l p{.65\textwidth} l}
    \toprule
    \multicolumn{2}{l}{\emph{Parallel Chip-Firing}} & \emph{Defined in} \\
    \midrule

    $\chips{v}{t}$ & Number of chips on vertex $v$ in position $\pos{t}$. &
    Section~\ref{pres} \vspace{\tablespace}\\

    $\pfp{v}$ & Periodic firing pattern of $v$. & Section~\ref{nonclumpiness}
    \vspace{\tablespace}\\

    $\firing{v}{t}$ & Indicates whether or not vertex $v$ fires in $\pos{t}$. &
    Section~\ref{pres} \vspace{\tablespace}\\

    $\receiving{v}{t}$ & Number of chips vertex $v$ will receive in $\pos{t}$.
    & Section~\ref{pres} \vspace{\tablespace}\\

    $\period$ & Period of $\s$. & Section~\ref{pres} \vspace{\tablespace}\\

    $\mots$ & Set of vertices that are motors in $\s$. & Section~\ref{motors}
    \vspace{\tablespace}\\

    \toprule
    \multicolumn{3}{l}{\emph{Graphs}} \\
    \midrule

    $\V$ & \multicolumn{2}{l}{Vertex set of graph $G$} \vspace{\tablespace}\\

    $\E$ & \multicolumn{2}{l}{Edge set of graph $G$} \vspace{\tablespace}\\

    $\N[G]{v}$ & \multicolumn{2}{l}{Neighbors of vertex $v$.}
    \vspace{\tablespace}\\

    $\deg[G]{v}$ & \multicolumn{2}{p{.815\textwidth}}{The degree of vertex $v$
      in graph $G$.} \vspace{\tablespace}\\

    \toprule
    \multicolumn{3}{l}{\emph{Other}} \\
    \midrule

    $[a,b]$ & \multicolumn{2}{l}{The integer interval $\{a, a+1, \dots, b\}$.}
    \vspace{2\tablespace}
  \end{tabular}\\
\end{centering}
\hidegame

We leave out the subscript $G$ or superscript $\s$ if there is no ambiguity.

\section{Motors}\label{motors}

Let $G$ be a graph. Suppose we wish to study the periodic behavior of games on
$G$, focusing on a particular subgraph $H \subseteq G$. Consider
\begin{equation*}
  X = \set{v \in \V \setminus \V[H]}{\N{v} \cap \V[H] \neq \emptyset},
\end{equation*}
the boundary of $H$. Knowing the initial chip configuration on $\V[H] \cup X$
is in general not enough to determine all subsequent configurations because
vertices in $X$ may have interactions with vertices outside of $\V[H] \cup
X$. However, we do know that every vertex assumes a pattern of firing and
waiting that repeats periodically as soon as a game reaches a periodic
position. Therefore, we can simulate the presence of the rest of $G$ by having
each vertex in $X$ fire with a regular pattern regardless of the number of
chips it receives.

The $\emph{firing sequence}$ of a vertex $v$ in game $\s$ is the sequence
$(\firing{v}{t})_{t \in \nats}$. A \emph{motorized parallel chip-firing game},
or simply ``motorized game'', on $G$ is a game $\s$ obeying \eqref{gameDef}
with a non-empty set of \emph{motors} $\mots \subseteq \V$. Each motor follows
a predetermined firing sequence, firing without regard for the normal rules of
the parallel chip-firing game, which means, for example, that a motor may have
a negative number of chips. Put another way, for each $m \in \mots$,
$\firing{m}{t}$ does not depend on $\chips{m}{t}$. The term ``ordinary game''
refers to a game with no motors when there is ambiguity. A motorized game is
shown in Figure~\ref{motorizedTreeNoGlider}.

\begin{centering}
\begin{figure}[tbh]
  \subfloat{\includegraphics[width=\figWidthA]{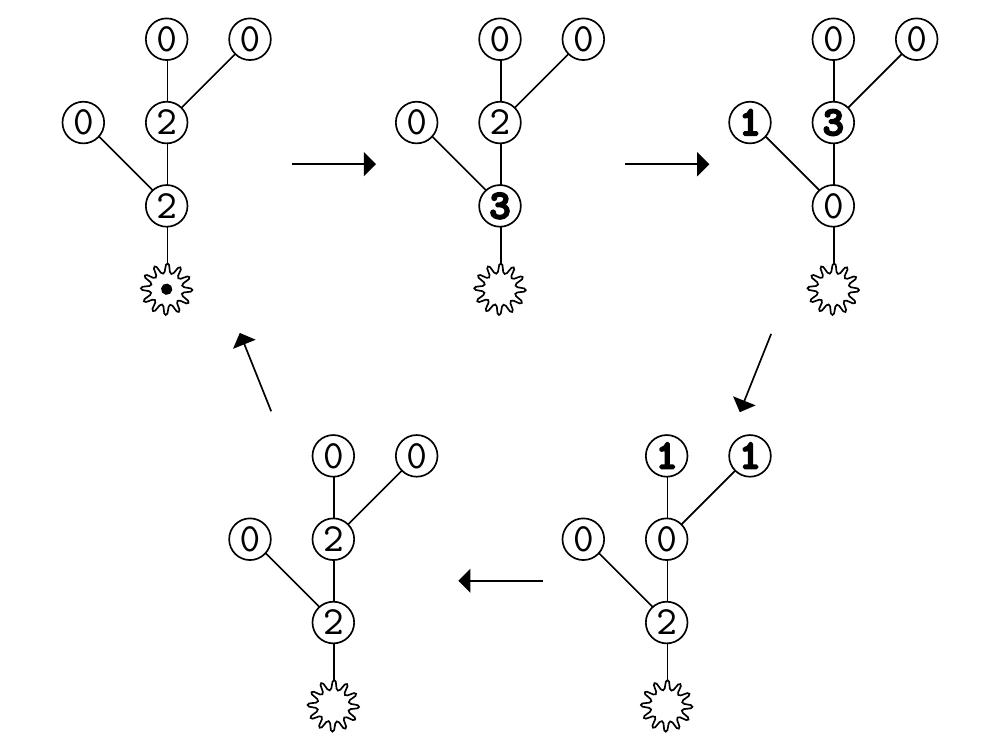}}
  \subfloat{\includegraphics[width=\figWidthB]{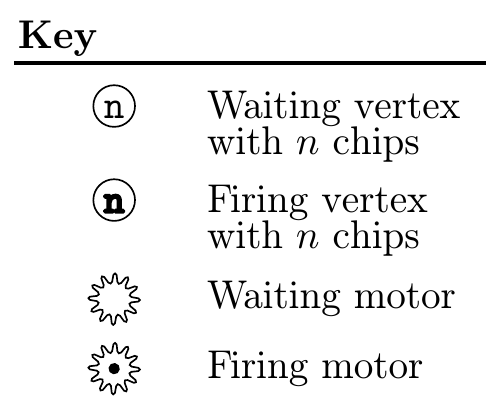}}
  \caption{A motorized parallel chip-firing game. The motor has firing sequence
    $(1,0,0,0,0,1,0,0,0,0,\dots)$.}
  \label{motorizedTreeNoGlider}
\end{figure}
\end{centering}

If a motorized game $\s$ is eventually periodic (which is the case if every
motor's firing sequence is eventually periodic), then just as in an ordinary
game, every vertex fires the same number of times each period.  The proof is
identical to the proof of this fact for ordinary games~\cite{jiang}: all
neighbors of the vertex that fires the most times each period must also fire
that maximal number of times, and by induction, so do all vertices. (Recall
that we consider in this paper only connected graphs.)

Let $f \in \{0,1\}$. Call an interval $[a, b]$ with $a < b$ an \emph{$f$-clump}
of $v \in \V$ if and only if $\firing{v}{t} = f$ for all $t \in [a, b]$. We
call $[a, b]$ an \emph{$f$-max-clump} if, in addition, $\firing{v}{a-1} =
\firing{v}{b+1} = 1-f$. Given $v \in \V$, we can express $\nats$ as the union
of max-clumps of $v$ and times during which $v$ alternates between firing and
waiting.

The proof of Theorem~\ref{cheapLunch} follows the same structure as the proof
that ordinary games on trees have period~1 or~2~\cite{bitarGoles}. In fact, we
rely on a lemma originally introduced for that proof.

\begin{lem}[{\cite[Lemma 1]{bitarGoles}}] \label{bitarGoles}
Let $\s$ be a game on $G$. For all $v \in \V$ and $f \in \{0,1\}$, if $[a, b]$
is an $f$-clump of $v$, then there exists a neighbor $w \in \N{v}$ such that
$[a-1, b-1]$ is an $f$-clump of $w$.
\end{lem}

Less technically, every clump of firing or waiting by a vertex must be
supported by at least one of its neighbors. The lemma follows from the
pigeonhole principle and Lemma~\ref{strongbg}, which we state and prove later.

\begin{thm} \label{cheapLunch}
Let $\s$ be a periodic motorized game on tree $T$. For all $v \in \V[T]$ and
$f \in \{0,1\}$, if $[a, b]$ is an $f$-clump of $v$, then $[a-D, b-D]$ is an
$f$-clump of $m$ for some $m \in \mots$, where $D$ is the distance from $m$ to
$v$.
\end{thm}

\begin{proof}
The result is clear if all vertices either always fire or always wait. In all
other cases, each firing sequence has a max-clump, and the argument is roughly
as follows. By Lemma~\ref{bitarGoles}, each clump of a vertex must be supported
by a clump of a neighbor.  Following the ``chain of support'' gives a sequence
of vertices that either is infinite or ends with a motor. If we consider the
containing max-clumps of clumps, we can guarantee a sequence with no
backtracking. Trees have no cycles, so the sequence must end with a motor. The
details follow.

Let $v_0 = v$ and $[a_0, b_0] \supseteq [a, b]$ be an $f$-max-clump of
$v_0$. By Lemma~\ref{bitarGoles}, given a vertex $v_i \not\in \mots$ with clump
$[a_i, b_i]$, we can pick a supporting vertex $v_{i+1} \in \N{v_i}$ and
integers $a_{i+1}$ and $b_{i+1}$ such that $[a_{i+1}, b_{i+1}]$ is an
$f$-max-clump of $v_i$ and $[a_i - 1, b_i - 1] \subseteq [a_{i+1},
b_{i+1}]$. (The fact that $\s$ is periodic means we need not worry about
negative turn numbers.) If there is a maximum $i$ for which $v_i$ exists, that
vertex must be a motor, which would mean $[a-D, b-D] \subseteq [a_D, b_D]$,
where $D$ is the maximum $i$ and $m = v_D \in \mots$. Thus, it suffices to show
that the sequence $(v_0, v_1, \ldots)$ eventually terminates. There are
finitely many vertices in the graph, so it suffices to show that the $v_i$ are
all distinct.

$T$ has no cycles, so if $v_i \neq v_{i+2}$ for all $i$, then all $v_i$ are
distinct. Suppose that $v_i = v_{i+2}$ for some $i$. Then $[a_i, b_i] \cup
[a_{i+2}, b_{i+2}]$ would be a clump of $v_i$. However, $[a_i - 2, b_i - 2]
\subseteq [a_{i+2}, b_{i+2}]$, making $[a_i - 2, b_i]$ a clump of $v_i$. But
$[a_i, b_i]$ is a max-clump for all $i$, so $v_i \neq v_{i-2}$ for all $i$.
\end{proof}

Call a firing sequence \emph{clumpy} if it contains two consecutive 0s and two
consecutive 1s; otherwise, call it \emph{nonclumpy}.

\begin{cor} \label{freeLunch}
Let $\s$ be a periodic motorized game on tree $T$ with a single motor $m$. If
$m$ has a nonclumpy firing sequence but has at least one clump, then
$\firing{v}{t+D} = \firing{m}{t}$ for all $v \in \V[T]$ and $t \in \nats$,
where $D$ is the distance from $v$ to $m$.
\end{cor}

\begin{proof}
The result is again clear in the always waiting and always firing cases. In all
other cases, $m$ has an $f$-max-clump, where $f \in \{0,1\}$. Let $v \in
\V[T]$. By Theorem~\ref{cheapLunch}, $v$ has a nonclumpy firing sequence
because $m$ does. All vertices fire the same number of times every
period~\cite[Proposition~2.5]{jiang}, so $v$ must have at least one max-clump,
again because $m$ does. For every $f$-max-clump $[a, b]$ of $v$, $[a-D, b-D]$
is an $f$-clump of $m$. The non-max-clump intervals of $v$'s firing sequence
are alternations between 0 and 1, starting and ending with $1-f$. The same must
be true of $m$ for it to fire the same number of times as $v$ each period.
\end{proof}

The reason we require the games in Theorem~\ref{cheapLunch} and
Corollary~\ref{freeLunch} to be periodic is to consider arbitrarily many past
turns. We can likely weaken this condition if we require the statements to be
true only after sufficiently many turns, though exactly how many turns that is
could depend on the activity (firing frequency; see~\cite{levine}) of the
motor, the size of the tree, and the total number of chips in the initial
position.

\section{Simulating Motors} \label{simulatingMotors}
\showgame

In this section, to refer to multiple chip-firing games unambiguously, we
include the subscripts and superscripts in, for example, $\deg[G]{v}$ and
$\firing{v}{t}$.

We call a firing sequence $(f_t)_{t \in \nats}$ \emph{possible} if there exists
an ordinary game $\s$ on some graph $G$ such that $\firing{v}{t} = f_{t}$ for
all $t \in \nats$. Our next theorem states that we can simulate motorized games
with ordinary games as long as every motor's firing sequence is
possible. Figure~\ref{natMot} demonstrates the concept.

\begin{centering}
\begin{figure}[tbh]
  \subfloat{\includegraphics[width=\figWidthA]{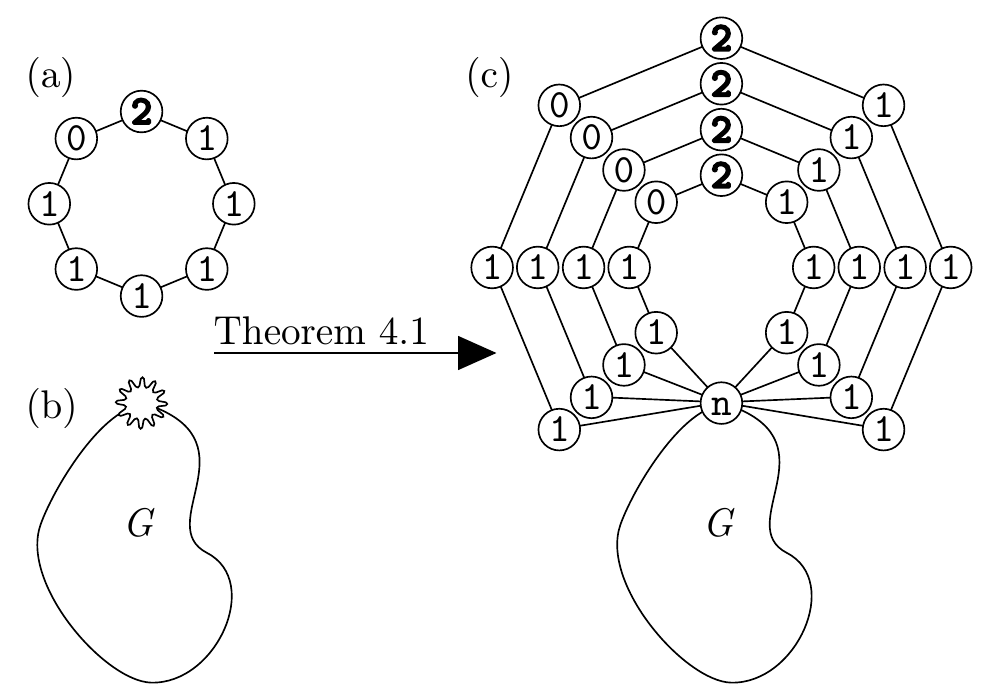}}
  \subfloat{\includegraphics[width=\figWidthB]{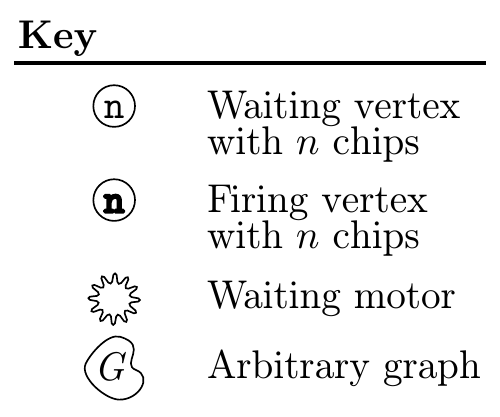}}
  \caption{Suppose the motor in motorized game (b) has firing sequence
    $(0,0,0,0,1,0,0,0,\dots)$. This occurs in ordinary game (a). By using
    sufficiently many copies of (a) and carefully choosing $n$, we construct
    (c). The behavior of $G$ in (c) is identical to the behavior of $G$ in
    (b).}
\label{natMot}
\end{figure}
\end{centering}

\begin{thm} \label{natMotors}
Let $\s$ be a motorized game on $G$. If every motor's firing sequence is
possible and has the same activity (firing frequency; see~\cite{levine}), then
there exists an ordinary game $\s'$ on a graph $H \supseteq G$ such that
\begin{itemize}
\item $\firing[\s']{u}{t} = \firing{u}{t}$ for all $t \in \nats$ and $u \in
  \V$,
\item $\deg[H]{v} = \deg[G]{v}$ for all $v \in \V \setminus \mots$, and
\item the subgraph of $H$ induced by $V(G)$ is $G$. (That is, $H$ contains no
  edges between vertices of $G$ that are not also in $G$.)
\end{itemize}
\end{thm}

\begin{proof}
Our approach will be, for each $m \in \mots$, to attach many copies of a graph
with a vertex with $m$'s firing sequence to $m$. If sufficiently many copies
are attached, the number of chips $m$ has due to its neighbors in $G$ becomes
irrelevant as to whether or not it fires.

For each $m \in \mots$, let $A_m$ be a graph such that there exists a game
$\s^m$ and some vertex $u _m\in \V[A_m]$ such that $\firing[\s^m]{u_m}{t} =
\firing{m}{t}$ for all $t \in \nats$. Let $a_m$ and $b_m$ be the minimum and
maximum respectively of $\set{\chips{m}{t}}{t \in \nats}$. These bounds exist
because all the motors have possible firing sequences with the same activity,
which means the motorized game is eventually periodic. Let $k_m = b_m - a_m +
1$, and let $H$ be the union of $G$ and $k_m$ copies of each $A_m$, with $G$
and the copies of $A_m$ disjoint except for $m = u_m$ for each $m \in \mots$.

It is clear by construction that $H$ contains no new edges between vertices of
$G$ and that
\begin{itemize}
\item $\deg[H]{m} = k_m\deg[{A_m}]{u_m} + \deg[G]{m}$ for all $m \in \mots$,
\item $\deg[H]{u} = \deg[{A_m}]{u}$ for all $u \in \V[A_m] \setminus \{m\}$ for
  each $m \in \mots$, and
\item $\deg[H]{v} = \deg[G]{v}$ for all $v \in \V \setminus \mots$.
\end{itemize}
Suppose that for some $t \in \nats$, $\pos[\s']{t}$ satisfies the following.
\begin{enumerate} \label{posAtT}
\item $\chips[\s']{m}{t} = k_m\chips[\s^m]{u_m}{t} + \deg[G]{m} + \chips{m}{t}
  - a_m$ for all $m \in \mots$.
\item $\chips[\s']{u}{t} = \chips[\s^m]{u}{t}$ for all $u \in \V[A_m] \setminus
  \{m\}$ for each $m \in \mots$.
\item $\chips[\s']{v}{t} = \chips{v}{t}$ for all $v \in \V \setminus \mots$.
\end{enumerate}
We will show that $\pos[\s']{t+1}$ satisfies the above as well. We have
$\deg[H]{v} = \deg[G]{v}$ for all $v \in \V \setminus \mots$, so
$\firing[\s']{v}{t} = \firing{v}{t}$ for all $v \in \V \setminus
\mots$. Similarly, $\firing[\s']{u}{t} = \firing[\s^m]{u}{t}$ for all $u \in
\V[A_m] \setminus \{m\}$ for each $m \in \mots$. Finally, for all $m \in
\mots$, if $\firing[\s^m]{u_m}{t} = 0$, then
\begin{align*}
  \chips[\s']{m}{t} &\leq k_m(\deg[{A_m}]{u_m}-1) + \deg[G]{m} + \chips{m}{t} -
  a_m \\
  &= k_m\deg[{A_m}]{u_m} + \deg[G]{m} + (\chips{m}{t} - b_m) - 1 \\
  &\leq \deg[H]{m} - 1,
\end{align*}
and if $\firing[\s^m]{u_m}{t} = 1$, then
\begin{align*}
  \chips[\s']{m}{t} &\geq k_m\deg[{A_m}]{u_m} + \deg[G]{m} + (\chips{m}{t} - a_m)
  \\
  &\geq \deg[H]{m},
\end{align*}
so $\firing[\s']{m}{t} = \firing[\s^m]{u_m}{t} = \firing{m}{t}$.

We now know that $\firing[\s']{v}{t} = \firing{v}{t}$ for all $v \in \V[H]$, so
$\chips[\s']{v}{t+1} = \chips{v}{t+1}$ for all $v \in \V[G] \setminus \mots$
and $\chips[\s']{u}{t+1} = \chips[\s^m]{u}{t+1}$ for all $u \in \V[A_m]
\setminus \{m\}$ for each $m \in \mots$. Finally, we have
\begin{align*}
  \chips[\s']{m}{t+1} &= k_m\chips[\s^m]{u_m}{t} + \deg[G]{m} + \chips{m}{t} -
  a_m + \receiving[\s']{m}{t} - \firing[\s']{m}{t}\deg[H]{m} \\
  &= k_m\chips[\s^m]{u_m}{t} + \deg[G]{m} + \chips{m}{t} - a_m +
  \receiving{m}{t}
  - \firing{m}{t}\deg[G]{m}\ + \\
  &\qquad k_m\receiving[\s^m]{u_m}{t} -
  k_m\firing[\s^m]{u_m}{t}\deg[{A_m}]{u_m} \\
  &= k_m(\chips[\s^m]{u_m}{t} + \receiving[\s^m]{u_m}{t} -
  \firing[\s^m]{u_m}{t}\deg[{A_m}]{u_m}) + \deg[G]{m}\ + \\
  &\qquad (\chips{m}{t} + \receiving[\s]{m}{t} - \firing[\s]{m}{t}\deg[G]{v)} -
  a_m \\
  &= k_m\chips[\s^m]{u_m}{t+1} + \deg[G]{m} + \chips{m}{t+1} - a_m.
\end{align*}
for all $m \in \mots$.

We can distribute chips in $\pos[\s']{0}$ such that it satisfies (1), (2), and
(3), in which case, by induction, $\pos[\s']{t}$ satisfies (1), (2), and (3)
for all $t \in \nats$, implying $\firing[\s']{u}{t} = \firing{u}{t}$ for all $v
\in V(G)$.
\end{proof}

In Theorem~\ref{cheapLunch}, motors were primarily a convenient intuition and
terminology; we could have proved a similar theorem within the context of the
ordinary parallel chip-firing game, though its statement would have been
messier. Theorem~\ref{natMotors} demonstrates another way in which the motor
concept is useful. Its constructive power makes certain conjectures easy to
prove or disprove by example. For instance, motors make it easy to construct
games in which the period isn't bounded by the number of vertices.

\hidegame

\section{Signed Sums of Binary Sequences}\label{binSeq}
We take a brief break from the parallel chip-firing game itself to consider
binary strings.  Throughout this section, $p$ and $q$ are length-$n$ binary
strings, $b \in \{0,1\}$, and $\ol{b} = 1-b$. We denote the $i$\xth element of
a binary string $p$ as $p_i$, and any integer equivalent to $i \bmod n$ may
replace $i$.

The following definition formalizes the notion of part of a string being
``mostly'' 0s or 1s. A \emph{$b$-sector} of $p$ is a $\subseteq$-maximal
integer interval $[x,y]$ such that $p_{y-1} = p_y = b$ and either $p_i = b$ or
$p_{i+1} = b$ for all $i \in [x,y]$. Informally, $b$-sectors end with
consecutive $b$s and extend back as far as possible. We have to make an
exception for always-alternating strings, such as 01010101, that have neither 0
nor 1 twice in a row. We arbitrarily define $[0, n-1]$ to be a 0-sector of
them. It is the transitions between 0- and 1-sectors that are ultimately
important, so this exception is acceptable. It is simple to verify that the
indices of every binary string have a unique decomposition into 0- and
1-sectors. An example is shown in Figure~\ref{sectorEx}.

\begin{figure}[tbh]
  \[
    \aunderbrace[l1r]{01000100}
    \!\aoverbrace[L1R]{11011}
    \!\aunderbrace[l1r]{00}
    \!\aoverbrace[L1R]{101011}
  \]
  \caption{A string's 0-sectors (marked below) and 1-sectors (marked
    above).}
  \label{sectorEx}
\end{figure}

Let
\begin{align*}
  s_i(p) &= \begin{cases}
    -1 & \text{if $i$ is in a 0-sector of $p$} \\
    1 & \text{if $i$ is in a 1-sector of $p$}
  \end{cases} \\
  \delta_i(p) &= \begin{cases}
    0 & \text{if $i$ is in a $b$-sector of $p$ and $i+1$ is in a $b$-sector of
      $p$} \\
    1 & \text{if $i$ is in a $b$-sector of $p$ and $i+1$ is in a
      $\ol{b}$-sector of $p$}
  \end{cases} \\
  M_i(p,q) &= s_i(p)(p_i - q_{i-1}) + s_i(q)(q_i - p_{i-1}) - \delta_i(p) -
  \delta_i(q).
\end{align*}
Our main theorem in this section concerns what we refer to as the
\emph{mischief between $p$ and $q$},
\[
  M(p,q) = \sum_{i=0}^{n-1} M_i(p,q),
\]
Mischief, superficially speaking, measures three things: how much $p$ differs
from $q$ delayed one time step, which we call the \emph{misbehavior of $q$
  towards $p$}; vice versa; and the total number of sector switches. An example
calculation is shown in Figure~\ref{moraleEx}. The rules of the parallel
chip-firing game put a global upper bound on the total disagreement between
vertices, yet the following theorem states that mischief is nonnegative,
meaning that sector switches require disagreement. We show in
Section~\ref{nonclumpiness} that this implies that firing sequences with sector
switches are impossible once a game has become periodic.

\newbracespec{d}{!{0.05em}@{\hspace{0.1em}}!{0.05em}}
\newbracespec{e}{!{0.05em}}
\begin{figure}[tbh]
  \begin{align*}
    p\colon &\aunderbrace[e*{12}{d}1r]{010001}_4
    \!\aoverbrace[L1*{8}{d}1R]{00110}^2
    \!\aunderbrace[l1@{\hspace{0.1em}}1r]{11}_2
    \!\aoverbrace[L1*{10}{d}1R]{001010}^2
    \!\aunderbrace[l1*{2}{d}e]{11}
    & p\colon &\aunderbrace[l1r]{01000100}_2
    \!\aoverbrace[L1R]{11011}^4
    \!\aunderbrace[l1r]{00}_0
    \!\aoverbrace[L1R]{101011}^4 \\
    q\colon &\aunderbrace[1r]{1000100}_2
    \!\aoverbrace[L1R]{11011}^4
    \!\aunderbrace[l1r]{00}_0
    \!\aoverbrace[L1R]{101011}^4
    \!\aunderbrace[l1]{0}
    & q\colon &\aunderbrace[e*{14}{d}1r]{1000100}_2
    \!\aoverbrace[L1*{8}{d}1R]{11011}^4
    \!\aunderbrace[l1@{\hspace{0.1em}}1r]{00}_0
    \!\aoverbrace[L1*{10}{d}1R]{101011}^4
    \!\aunderbrace[l1]{0}
  \end{align*}
  \caption{We calculate the mischief between two strings by dividing each into
    sectors. The misbehavior of $p$ towards $q$ is $- (2-4) + (4-2) - (0 - 2) +
    (4 - 2) = 8$, as shown on the left. On the right, we see that the
    misbehavior of $q$ towards $p$ is 0. Each string has 4 sector switches, so
    $M(p,q) = 8 + 0 - 4 - 4 = 0$.}
  \label{moraleEx}
\end{figure}

\begin{thm}\label{morale}
The mischief between any two binary strings of equal length is nonnegative.
\end{thm}

\begin{proof}
Let
\[
  \mu_i(p,q) = (p_{i-1},p_i,s_i(p),s_{i+1}(p),q_{i-1},q_i,s_i(q),s_{i+1}(q)),
\]
a tuple of all information required to calculate $M_i(p,q)$, and let $\G$ be a
weighted digraph with
\begin{align*}
  \V[\G] &= \set{\mu_i(p,q)}{p,q \text{ strings}, i \in \nats} \\
  \E[\G] &= \set{(\mu_i(p,q),\mu_{i+1}(p,q),M_i(p,q))}{p,q \text{ strings}, i
    \in \nats}
\end{align*}
(The third item of each edge is its weight.) Define the weight of a path to be
the sum of the weights of its member edges, and call a path negative if it has
negative weight. The mischief $M(p,q)$ is the weight of a path in $\G$ induced
by the sequence of vertices $(\mu_0(p,q), \dots, \mu_{n-1}(p,q),
\mu_0(p,q))$. Therefore, it suffices to show that $\G$ has no negative
cycles---in particular, no negative cycles of length $n$.

We discuss below the constraints that define what tuples can exist as some
$\mu_i(p,q)$. These constraints yield a graph $\G$ with 64 vertices and 256
edges, which is impractical to analyze by hand. However, running the
Bellman-Ford algorithm~\cite{bellmanford} on $\G$ shows it to be free of
negative cycles. A Python program specifies both the graph and the algorithm in
detail in the appendix.
\end{proof}

The only constraint on edges of $\G$ is that the state shared by the connected
vertices be consistent. For example, both $\mu_i(p,q)$ and $\mu_{i+1}(p,q)$
contain $p_i$. There are two constraints on the vertices of $\G$ enforced by
the definition of sectors in $p$. Analogous constraints apply for $q$.
\begin{itemize}
\item If $p_{i-1} = p_i = b$, then $s_i(p)$ indicates that $i$ is in a
  $b$-sector of $p$.
\item If $s_i(p) \neq s_{i+1}(p)$, then $p_{i-1} = p_i = b$ such that $s_i(p)$
  indicates that $i$ is in a $b$-sector of $p$.
\end{itemize}
To get a sense for how these force nonnegative mischief, consider the latter
constraint. The sector switch reduces mischief, but it also makes $p_{i+1} =
\ol{b}$. If $q_i = b$, then $q$ misbehaves towards $p$ at $i+1$. If $q_i =
\ol{b}$ and $i$ is in a $\ol{b}$-sector of $q$, then $p$ misbehaves towards $q$
at $i$. However, we cannot finish the last case, $q_i = \ol{b}$ with $i$ in a
$b$-sector of $q$, with similar local reasoning. The former constraint limits
the frequency with which that last case can happen, but there are many
scenarios to consider before dispatching the case. Using the Bellman-Ford
algorithm avoids this further casework.

\section{Nonclumpiness of Periodic Firing Patterns}\label{nonclumpiness}
We consider parallel chip-firing game $\s$ on undirected graph $G$. The
\emph{periodic firing pattern} (PFP) of a vertex $v \in \V$ is the binary
string
\[
  \firing{v}{t_0}\dots\firing{v}{t_0 + \period - 1},
\]
where $t_0$ is the smallest natural number such that $\pos{t_0}$ is
periodic\footnote{The reason we introduce PFPs instead of continuing to reason
  with firing sequences is because a PFP is aware of the period of the game it
  occurs in. For instance, the PFPs 01 and 0101 result in the same periodic
  firing sequence, but while the latter, which has period 4, might occur in the
  same game as the PFP 0011, the former, which has period 2, cannot.}. We write
the PFP of $v$ as $\pfp{v}$. For simplicity, we assume here that $t_0 = 0$ and
index PFPs modulo $\period$.

Let $\pfps$ be the set of all PFPs that are compatible with $\s$, which means
they have the same number of 0s and 1s as a PFP that occurs in $\s$. Call a PFP
with both consecutive 0s and consecutive 1s \emph{clumpy}, and let $\cpfps$ be
the set of clumpy PFPs that are compatible with $\s$. (Recall that the
$\period$\xth and 0\xth entries of a PFP are the same, so, for example, 011010
is clumpy.) It is known that, given almost any\footnote{The given construction
  requires that the PFP not be decomposable to a repeated substring. Using
  Theorem~\ref{natMotors}, one can expand the construction to any nonclumpy PFP
  other than those that are 01 or 10 repeated more than once. These PFPs turn
  out to be impossible, though the corresponding firing sequences are possible
  in games of period 2.}  nonclumpy PFP, one can construct a parallel
chip-firing game on a simple cycle in which every vertex has that PFP shifted
by some number of steps~\cite{cycle}. We prove here that clumpy PFPs cannot
occur in any parallel chip-firing game.

\begin{lem}\label{strongbg}
In a periodic game on $G$, for all $v \in \V$ and $a, b \in \nats$,
\begin{equation}\label{local}
  -\deg{v} + 1 \leq \sum_{t=a}^{b}(\receiving{v}{t-1} - \deg{v}\firing{v}{t})
  \leq \deg{v}-1.
\end{equation}
\end{lem}

\begin{proof}
We express $\chips{v}{b}$ in terms of $\chips{v}{a-1}$:
\begin{align*}
  \chips{v}{b} &= \chips{v}{a-1} +
  \smashoperator{\sum_{t=a-1}}^{b-1}(\receiving{v}{t} - \deg{v}\firing{v}{t})
  \\
  \chips{v}{b} - \deg{v}\firing{v}{b} &= \chips{v}{a-1} -
  \deg{v}\firing{v}{a-1} + \sum_{t=a}^{b}(\receiving{v}{t-1} -
  \deg{v}\firing{v}{t}).
\end{align*}
As mentioned in \cite[Section 2]{kominers}, the set of vertices $v$ such that
$\chips{v}{t} \geq 2\deg{v}$ chips is nonincreasing as $t$ increases. In
particular, in a periodic game, this set is empty unless the game has period 1,
in which case the lemma is clear. Otherwise, simple analysis of the waiting and
firing cases shows that $\chips{v}{t} \leq 2\deg{v} - 1$ implies $0 \leq
\chips{v}{t} - \deg{v}\firing{v}{t} \leq \deg{v} - 1$, which gives the desired
inequality.
\end{proof}

We define
\begin{align*}
\pfpverts{p} &= {\set{v \in \V}{\pfp{v}=p}} \\
\pfpedges{p}{q} &= {\set{\{v,w\} \in \E}{\pfp{v}=p, \pfp{w} = q}} \\
\misb{S}{p}{q} &= \sum_{i \in S}(p_{i} - q_{i-1})
\end{align*}
for binary strings $p$ and $q$ and sets of indices $S \subseteq \nats$. Note
that $\misb{S}{p}{q}$ is very closely related to the misbehavior of $q$ towards
$p$, missing only sector-dependent signs. We now have the tools to prove our
main result.

\begin{thm}\label{nct}
Clumpy PFPs do not occur in the parallel chip-firing game.
\end{thm}

\begin{proof}
The key is to examine the mischief between all pairs of neighbor vertices in
which at least one neighbor has a clumpy PFP. Roughly speaking, summing an
inequality derived from Lemma~\ref{strongbg} over all vertices with clumpy PFPs
bounds a negative sum of mischiefs below, and summing the inequality given by
the Theorem~\ref{morale} over all edges incident with a vertex with a clumpy
PFP gives an upper bound on the same sum of mischiefs. The lower bound relates
positively with the number of vertices with clumpy PFPs, and the upper bound is
0.

Let $a,b \in \nats$ and $v \in \V$. Grouping the sum in \eqref{local} by $v$'s
neighbors instead of time steps yields
\[
  -\deg{v} + 1 \leq -\smashoperator{\sum_{w \in
      \N{v}}}\misb{[a,b]}{\pfp{v}}{\pfp{w}} \leq \deg{v} - 1.
\]
Reordering the inequality on the left, recalling that $\deg{v} = \size{\N{v}}$,
gives
\[
  1 \leq \smashoperator{\sum_{w \in \N{v}}}(1 - \misb{[a,b]}{\pfp{v}}{\pfp{w}}),
\]
and by considering the right inequality we obtain the same result with the $-$
switched to $+$. Let $p$ be a PFP. By summing the above over $v \in
\pfpverts{p}$ for some choice of signs, we obtain
\[
  \size{\pfpverts{p}} \leq \smashoperator{\sum_{\substack{v \in \pfpverts{p}
        \\ w \in \N{v}}}}(1 + r_v\misb{[a,b]}{p}{\pfp{w}}),
\]
where each $r_v = \pm1$ may be chosen dependent on $v$. (Notation: domains for
outer sums are above domains for inner sums.) That each
$\misb{[a,b]}{p}{\pfp{w}}$ term is preceded by a sign of our choice is a hint
that we can relate this quantity to misbehavior.

For all PFPs $p$, let $\sctr{p}$ be the set of sectors of $p$. Because each
sector of a PFP is of the form $[a,b]$ for some $a,b \in \nats$, we can sum the
above inequality over $X \in \sctr{p}$ and $p \in \cpfps$ to get
\begin{equation}\label{almost}
  \smashoperator[r]{\sum_{\substack{
        p \in \cpfps \\
        X \in \sctr{p}
  }}}\size{\pfpverts{p}} \leq
  \smashoperator{\sum_{\substack{
        p \in \cpfps \\ v \in \pfpverts{p} \\
        w \in \N{v} \\ X \in \sctr{p}
  }}}(1 + r_{v,X}\misb{X}{p}{\pfp{w}}),
\end{equation}
where each $r_{v,X} = \pm1$ may be chosen dependent on $v$ and $X$.

Let $p$ be a clumpy PFP. If $q$ is a clumpy PFP, then
\begin{equation}\label{misbSub1}
  M(p,q) =
  \smashoperator{\sum_{X \in \sctr{p}}}(s_X(p)\misb{X}{p}{q} - 1) +
  \smashoperator{\sum_{X \in \sctr{q}}}(s_X(q)\misb{X}{q}{p} - 1).
\end{equation}
(Abusing notation slightly, we write $s_X(p)$ instead of $s_i(p)$ if $i \in X
\in \sctr{p}$.) The $-1$ in each sum accounts for the
$-\delta_i(p)-\delta_i(q)$ term in $M_i(p,q)$. If instead $q$ is not clumpy,
then $\sctr{q} = \{[0,\period-1]\}$, so $q$ has no sector switches and
\begin{equation}\label{misbSub2}
  M(p,q) =
  \smashoperator{\sum_{X \in \sctr{p}}}(s_X(p)\misb{X}{p}{q} - 1) +
  s_{[0,\period-1]}(q)\misb{[0,\period-1]}{q}{p}.
\end{equation}
However, $\misb{[0,\period-1]}{q}{p} = 0$ because $p$ and $q$ have the same
length and number of 1s. This means that the mischief between $p$ and $q$ only
depends on the misbehavior of $q$ towards $p$ and sector switches in $p$. This
is important because Lemma~\ref{strongbg} bounds total misbehavior towards a
vertex from all neighbors.

Let $W = \set{v \in \V}{\pfp{v} \in \cpfps}$ be the set of vertices with clumpy
PFPs. Choosing $r_{v,X} = -s_X(p)$ in \eqref{almost} yields
\begin{align*}
  \smashoperator[r]{\sum_{\substack{
        p \in \cpfps \\
        X \in \sctr{p}
      }}}\size{\pfpverts{p}}
  &\leq
  \smashoperator{\sum_{\substack{
        p \in \cpfps \\ v \in \pfpverts{p} \\
        w \in \N{v} \\ X \in \sctr{p}
      }}}(1-s_X(p)\misb{X}{p}{\pfp{w}}) \\
  &=
  \smashoperator{\sum_{\substack{
        p \in \cpfps \\ v \in \pfpverts{p} \\
        w \in \N{v} \cap W \\ X \in \sctr{p}
      }}}(1-s_X(p)\misb{X}{p}{\pfp{w}}) +
  \smashoperator{\sum_{\substack{
        p \in \cpfps \\ v \in \pfpverts{p} \\
        w \in \N{v} \setminus W \\ X \in \sctr{p}
      }}}(1-s_X(p)\misb{X}{p}{\pfp{w}}) \\
  &=
  \smashoperator[l]{\sum_{\substack{
        p,q \in \cpfps \\
        e \in \pfpedges{p}{q}
      }}}\!\!
  \Bigg(\smashoperator[r]{\sum_{X \in \sctr{p}}}(1-s_X(p)\misb{X}{p}{q})
  + \smashoperator{\sum_{X \in \sctr{q}}}(1-s_X(q)\misb{X}{q}{p})\Bigg)
  \\ &\qquad +
  \smashoperator{\sum_{\substack{
        p \in \cpfps \\ v \in \pfpverts{p} \\
        w \in \N{v} \setminus W \\ X \in \sctr{p}
      }}}(1-s_X(p)\misb{X}{p}{\pfp{w}}). \\
\end{align*}
Note that we consider $p$ and $q$ in the sum over $p,q \in \cpfps$ to be
unordered. (One alternative notation is a sum over $\{p,q\} \subseteq \cpfps$.)
We now substitute using \eqref{misbSub1} and \eqref{misbSub2} and apply
Theorem~\ref{morale} to get
\[
  \smashoperator[r]{\sum_{\substack{
        p \in \cpfps \\
        X \in \sctr{p}
      }}}\size{\pfpverts{p}}
  \leq
  -\smashoperator{\sum_{\substack{
        p,q \in \cpfps \\
        e \in \pfpedges{p}{q}
      }}}M(p,q) -
  \smashoperator{\sum_{\substack{
        p \in \cpfps \\
        v \in \pfpverts{p} \\ w \in \N{v} \setminus W
      }}}M(p,F(w))
  \leq 0.
\]
Sets have nonnegative sizes, so $\size{\pfpverts{p}} = 0$ for all $p \in
\cpfps$.
\end{proof}

\section{Implications of Nonclumpiness} \label{corollaries}
It is a basic property of the parallel chip-firing game that every vertex fires
the same number of times each period~\cite{jiang}. This means, roughly
speaking, that every periodic game is either ``mostly waiting'' with bursts of
firing or ``mostly firing'' with bursts of waiting. (In fact, there is a
bijection between these two types of games. Each periodic game has a complement
that inverts firing and waiting~\cite{jiang}.) This is because if a vertex
waits twice in a row, then because it therefore never fires twice in a row, it
fires less than half the time over the course of a period. Similarly, a vertex
that fires twice in a row fires more than half the time. We cannot have a
vertex that waits twice in a row and a vertex that fires twice in a row in the
same periodic game because each vertex fires the same number of times each
period.

\begin{cor}
Once a parallel chip-firing game is periodic, either no vertex fires twice in a
row or no vertex waits twice in a row.
\end{cor}

That is, in periodic games, a firing sequence is possible if and only if it is
nonclumpy.

The \emph{interior} of a set of vertices $W$ is $\set{v \in W}{N(v) \subseteq
  W}$. Because a waiting (or firing) vertex with only waiting (or firing)
neighbors will wait (or fire) the following turn as well, the above observation
proves the following conjecture of Fey and Levine~\cite{privateComms}.

\begin{cor}\label{feyLevine}
Once a parallel chip-firing game is periodic, the interior of the set of
waiting vertices is always empty, the interior of set of firing vertices is
always empty, or both interiors are always empty.
\end{cor}

Interestingly, Corollary~\ref{feyLevine} also implies Theorem~\ref{nct}. If
clumpy PFPs were possible, then a leaf attached to a motor with a clumpy PFP
would be at different times in both the waiting and firing interiors.

In one of the first papers on the parallel chip-firing game, Bitar and Goles
characterized parallel chip-firing games on trees~\cite{bitarGoles}.
Corollary~\ref{freeLunch} and Theorem~\ref{nct} allow us to characterize the
behavior on tree-like subgraphs---subgraphs that, if an edge to a root vertex
is cut, become a tree separated from the rest of the graph---by making the root
vertex a motor.

\begin{cor}
Let $\s$ be a periodic game on $G$ with period at least 3 in which no vertex
fires twice in a row, $H$ be a tree-like subgraph of $G$, and $m \in \V[H]$ be
the root of $H$. Then for all $v \in \V[H]$,
\[
  \chips{v}{t} = \begin{cases}
    \deg{v} & \textnormal{if } \firing{m}{t-D} = 1 \\
    0 & \textnormal{if } \firing{m}{t-D-1} = 1 \\
    \deg{v} - 1 & \textnormal{otherwise},
  \end{cases}
\]
where $D$ is the distance from $m$ to $v$. An analogous result holds if no
vertex waits twice in a row.
\end{cor}

In some sense, tree-like subgraphs are passive in that their vertices fire only
in response to their root-side neighbor firing. In a periodic game, we can
completely remove tree-like subgraphs without affecting the PFPs of the other
vertices.

\begin{cor}
Let $\s$ be a periodic game on $G$, leaf $l \in \V$ have single neighbor $m$,
and $G'$ be $G$ with $l$ deleted. Then a game $\s'$ exists on $G'$ with the
same firing behavior as $\s$.
\end{cor}

The starting position $\pos[\s']{0}$ can agree with $\pos{0}$ completely except
for possibly removing a chip from $m$. We consider the case where no vertex
fires twice in a row. Compared to $\s'$, vertex $m$ has to have an extra chip
to fire in $\s$.  However, unless $m$ fired the previous turn---which, because
$l$ is a leaf, is equivalent to saying $l$ is firing this turn---$m$ will have
received the extra chip back from $l$, so removing both $l$ and the chip has no
effect on $m$ as long as $m$ does not fire while $l$ has a chip, which doesn't
happen due to nonclumpiness. The case where no vertex waits twice in a row is
analogous. This corollary concerns a leaf, though the result generalizes to all
tree-like subgraphs by repeated application, providing an alternate proof of
Corollary~\ref{freeLunch}.

\section{Discussion and Directions for Future Work} \label{discussion}

We have introduced motors, studied motorized games on trees, and shown that
motor-like behavior can be constructed in ordinary games, provided that each
motor has a possible firing sequence. We then showed that periodic firing
patterns are possible if and only if they are nonclumpy, which, among other
things, allows classification of periodic games as ``mostly waiting'' or
``mostly firing'' and the removal of tree-like subgraphs without loss of
generality.

We might expect the space of motorized games to be larger than that of ordinary
games. Theorem~\ref{natMotors} shows us that, as long as the firing sequences
involved are possible, the parallel chip-firing game is in some sense just as
``expressive'' as its motorized variant. This allows, for example, the
simulation of some aspects of the \emph{dollar game}, a variant of the general
chip-firing game discussed by Biggs~\cite{biggs}. In the dollar game, exactly
one vertex, the ``government'', may have a negative number of chips and fires
if and only if no other vertices can fire. We can construct a motorized
parallel chip-firing game in which we replace the government with a motor that
waits a sufficiently large number of steps between each firing such that it
never fires in the same step as another vertex. Biggs showed that every dollar
game tends towards a critical position regardless of the order of vertex
firings, so this motorized parallel chip-firing game tends towards the same
critical position. Theorem~\ref{natMotors} may help reveal the extent to which
the parallel chip-firing game can simulate additional aspects of the dollar
game and other general chip-firing games.

Despite the expressiveness we get due to motors, the nonclumpiness of firing
patterns tells us that the parallel chip-firing game is ``easier'' than its
rules explicitly tell us it must be. In addition to results mentioned in
Section~\ref{corollaries}, Theorem~\ref{nct} is a step towards reducing the
parallel chip-firing game to one of interacting ``gliders''. For example,
consider the situation in Corollary~\ref{freeLunch}. Intuitively, we can think
of this corollary as stating that each firing of the motor creates a wave of
gliders that travels away from the motor. In fact, the corollary, together with
Theorem~\ref{nct}, implies that every periodic position on tree-like subgraphs
must be the result of such gliders, providing a new test that can diagnose some
positions that are never repeated. Every game on a simple cycle with period at
least 3 can be described by gliders~\cite{cycle}. (See Figure~\ref{cycleFig}.)
We believe that this approach could be used to analyze periodic behavior of
games on further classes of graphs, such as those in which each vertex is in at
most once cycle.

\begin{centering}
\begin{figure}[tbh]
  \subfloat{\includegraphics[width=\figWidthA]{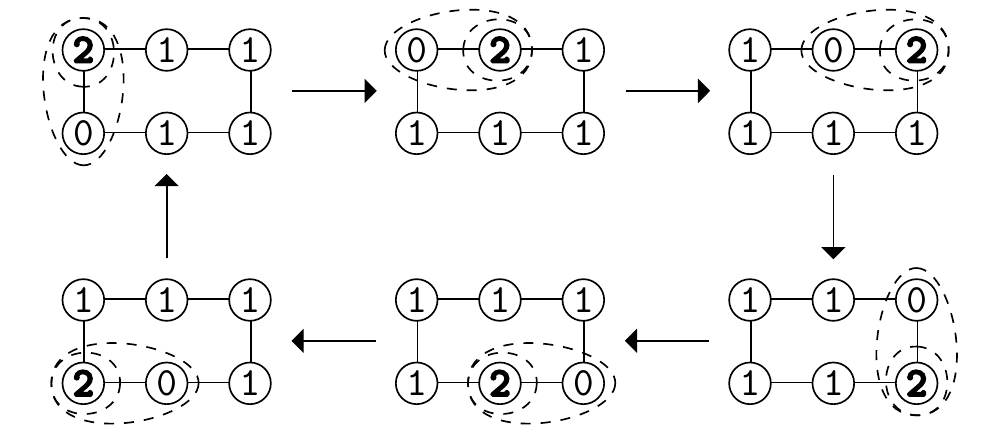}}
  \subfloat{\includegraphics[width=\figWidthB]{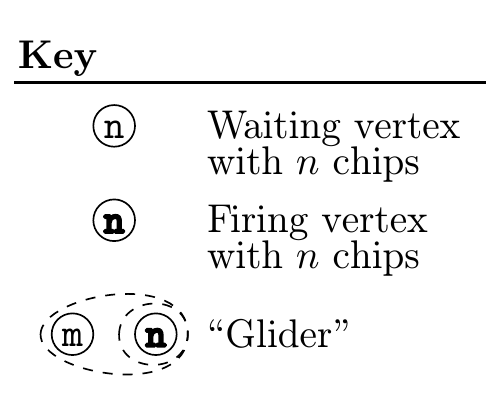}}
  \caption{A game on a 6-cycle in which a glider orbits once each period.}
  \label{cycleFig}
\end{figure}
\end{centering}

Nonclumpiness is essentially an unwritten rule of periods in the parallel
chip-firing game, which is unusual because no local property of the firing
mechanic disallows clumpiness. By contrast, in other graph automata that are
more restrictive than the parallel chip-firing game, such as \emph{source
  reversal}~\cite{sourceReversal} (essentially a parallel chip-firing game with
exactly one chip bound to each edge), nonclumpiness is obvious, even
locally. In the other direction, motors make it simple to show that certain
stronger restrictions do not apply to the parallel chip-firing game. For
example, a path where the leaves are motors can yield a game in which some
chips cannot be bound to a single edge, which is a property of source
reversal. We might ask which restrictions apply to which chip-firing-style
games. Is the parallel chip-firing game on undirected graphs the most general
game to which an analogue of Theorem~\ref{nct} applies?

We hope that the intuition and constructive powers of motors and the reduction
in the space of possible periodic games provided by nonclumpiness prove useful
in further research. \FloatBarrier

\appendix
\section*{The Bellman-Ford Algorithm} \label{bfAlg}
The following Python program below constructs $\G$ from Section~\ref{binSeq}
and shows it has no negative cycles. For a more detailed exposition of the
Bellman-Ford algorithm and a proof of its validity, see~\cite{bellmanford}.

\bigskip
\footnotesize
\begin{verbatim}
infty = float("inf")

class Vertex:
    def __init__(self):
        self.distance = infty

class Edge:
    def __init__(self, tail, head, weight):
        self.tail = tail
        self.head = head
        self.weight = weight

# Returns True if graph has a negative cycle or unreachable vertices.
def bellman_ford(verts, edges):
    # For our purposes, the start vertex is arbitrary.
    edges[0].tail.distance = 0
    # Find minimum length from the start vertex to each other vertex.
    for i in range(len(verts) - 1):
        for e in edges:
            new_distance = e.tail.distance + e.weight
            if new_distance < e.head.distance:
                e.head.distance = new_distance
    # Find triangle inequality failures caused by negative cycles.
    # Also confirms we searched entire graph (infty < infty == False).
    for e in edges:
        if e.tail.distance + e.weight < e.head.distance:
            return True
    return False

# Throughout, we use 0 instead of -1 for s_i(p), correcting when needed.

# Checks if mu_i(p,q) is compatible with the definition of a sector.
def valid_vert(mu1):
    (p0, p1, s1, s2, q0, q1, t1, t2) = mu1
    return (p0==s1 if p0==p1 else s1==s2) and (q0==t1 if q0==q1 else t1==t2)

# Checks if the first state can be followed by the second.
def valid_edge(mu1, mu2):
    (p0a, p1a, s1a, s2a, q0a, q1a, t1a, t2a) = mu1
    (p1b, p2b, s2b, s3b, q1b, q2b, t2b, t3b) = mu2
    return p1a==p1b and q1a==q1b and s2a==s2b and t2a==t2b

# Calculates M_i(p,q)
def weight(mu1):
    (p0, p1, s1, s2, q0, q1, t1, t2) = mu1
    return (2*s1-1)*(p1-q0) + (2*t1-1)*(q1-p0) - abs(s1-s2) - abs(t1-t2)

bits = {x : tuple((x // 2**i) % 2 for i in range(8)) for x in range(2**8)}
V = {bits[i]: Vertex() for i in range(2**8) if valid_vert(bits[i])}
E = [Edge(V[u], V[v], weight(u)) for u in V for v in V if valid_edge(u, v)]

# And Theorem 5.1 is...
print(not bellman_ford(V, E))
\end{verbatim}
\normalsize

\section*{Acknowledgements}
We thank the MIT PRIMES program for supporting the research that brought us
together. We thank Tiankai~Liu for improving upon a previous proof of
Theorem~\ref{morale}. We would also like to thank \mbox{Anne Fey}, \mbox{Lionel
  Levine} and \mbox{Tanya Khovanova} for helpful discussions. Finally, we thank
the anonymous referees who helped greatly improve the readability of the
paper. \mbox{Yan X Zhang} was supported by an NSF Graduate Fellowship.

\bibliographystyle{amsmodded}
\bibliography{refs}
\end{document}